\documentclass[12pt]{amsart}
\usepackage{amsmath,amssymb,amsbsy,amsfonts,latexsym,amsopn,amstext,cite,
                                               amsxtra,euscript,amscd,bm,mathabx}
\usepackage{url}
\usepackage[colorlinks,linkcolor=blue,anchorcolor=blue,citecolor=blue,backref=page]{hyperref}
\usepackage{color}
\usepackage{graphics,epsfig}
\usepackage{graphicx}
\usepackage{float} 
\usepackage[english]{babel}
\usepackage{mathtools}
\usepackage{todonotes}
\usepackage{url}
\usepackage[colorlinks,linkcolor=blue,anchorcolor=blue,citecolor=blue,backref=page]{hyperref}
\usepackage[norefs,nocites]{refcheck}

\hypersetup{breaklinks=true}

\usepackage[norefs,nocites]{refcheck}
\usepackage[english]{babel}
\begin{document}

\newtheorem{thm}{Theorem}
\newtheorem{lem}[thm]{Lemma}
\newtheorem{claim}[thm]{Claim}
\newtheorem{cor}[thm]{Corollary}
\newtheorem{prop}[thm]{Proposition} 
\newtheorem{definition}[thm]{Definition}
\newtheorem{rem}[thm]{Remark} 
\newtheorem{question}[thm]{Question}
\newtheorem{conj}[thm]{Conjecture}
\newtheorem{prob}{Problem}

\newtheorem{lemma}[thm]{Lemma}

\newcommand{\GL}{\operatorname{GL}}
\newcommand{\SL}{\operatorname{SL}}
\newcommand{\lcm}{\operatorname{lcm}}
\newcommand{\ord}{\operatorname{ord}}
\newcommand{\Op}{\operatorname{Op}}
\newcommand{\Tr}{\operatorname{Tr}}
\newcommand{\Nm}{\operatorname{Nm}}
\newcommand{\discr}{\operatorname{disrc}}
\newcommand{\supp}{\operatorname{supp}}

\numberwithin{equation}{section}
\numberwithin{thm}{section}
\numberwithin{table}{section}

\def\sssum{\mathop{\sum\!\sum\!\sum}}
\def\ssum{\mathop{\sum\ldots \sum}}
\def\iint{\mathop{\int\ldots \int}}

\def\vol {{\mathrm{vol\,}}}
\def\squareforqed{\hbox{\rlap{$\sqcap$}$\sqcup$}}
\def\qed{\ifmmode\squareforqed\else{\unskip\nobreak\hfil
\penalty50\hskip1em\null\nobreak\hfil\squareforqed
\parfillskip=0pt\finalhyphendemerits=0\endgraf}\fi}

\def \balpha{\bm{\alpha}}
\def \bbeta{\bm{\beta}}
\def \bgamma{\bm{\gamma}}
\def \blambda{\bm{\lambda}}
\def \bchi{\bm{\chi}}
\def \bphi{\bm{\varphi}}
\def \bpsi{\bm{\psi}}
\def \bomega{\bm{\omega}}
\def \btheta{\bm{\vartheta}}

\newcommand\veca{\boldsymbol{a}}
\newcommand\vecb{\boldsymbol{b}}
\newcommand\vech{\boldsymbol{h}}
\newcommand\vecq{\boldsymbol{q}}
\newcommand\vecu{\boldsymbol{u}}
\newcommand\vecv{\boldsymbol{v}}
\newcommand\vecw{\boldsymbol{w}}
\newcommand\vecx{\boldsymbol{x}}
\newcommand\vecy{\boldsymbol{y}}
\newcommand\vecz{\boldsymbol{z}}

\newcommand{\bfxi}{{\boldsymbol{\xi}}}
\newcommand{\bfrho}{{\boldsymbol{\rho}}}

\def\sfJ{\mathsf {J}}
\def\sfG {\mathsf {G}}
\def\sfK {\mathsf {K}}
\def\sfS {\mathsf {S}}
\def\sfT{\mathsf {T}}

 \def \xbar{\overline x}
  \def \ybar{\overline y}

\def\cA{{\mathcal A}}
\def\cB{{\mathcal B}}
\def\cC{{\mathcal C}}
\def\cD{{\mathcal D}}
\def\cE{{\mathcal E}}
\def\cF{{\mathcal F}}
\def\cG{{\mathcal G}}
\def\cH{{\mathcal H}}
\def\cI{{\mathcal I}}
\def\cJ{{\mathcal J}}
\def\cK{{\mathcal K}}
\def\cL{{\mathcal L}}
\def\cM{{\mathcal M}}
\def\cN{{\mathcal N}}
\def\cO{{\mathcal O}}
\def\cP{{\mathcal P}}
\def\cQ{{\mathcal Q}}
\def\cR{{\mathcal R}}
\def\cS{{\mathcal S}}
\def\cT{{\mathcal T}}
\def\cU{{\mathcal U}}
\def\cV{{\mathcal V}}
\def\cW{{\mathcal W}}
\def\cX{{\mathcal X}}
\def\cY{{\mathcal Y}}
\def\cZ{{\mathcal Z}}
\def\Ker{{\mathrm{Ker}}}

\def\NmQR{N(m;Q,R)}
\def\VmQR{\cV(m;Q,R)}

\def\Xm{\cX_m}

\def \A {{\mathbb A}}
\def \B {{\mathbb A}}
\def \C {{\mathbb C}}
\def \F {{\mathbb F}}
\def \G {{\mathbb G}}
\def \L {{\mathbb L}}
\def \K {{\mathbb K}}
\def \N {{\mathbb N}}
\def \PP {{\mathbb P}}
\def \Q {{\mathbb Q}}
\def \R {{\mathbb R}}
\def \Z {{\mathbb Z}}

\def \fU{\mathfrak U}
\def \fS{\mathfrak S}
\def \fM{\mathfrak M}
\def \fI{\mathfrak I}
\def \fJ{\mathfrak J}
\def \fH{\mathfrak H}
\def \fF{\mathfrak F}
\def \fE{\mathfrak E}
\def \fB{\mathfrak B}
\def \fA{\mathfrak A}

\def\GL{\operatorname{GL}}
\def\SL{\operatorname{SL}}
\def\PGL{\operatorname{PGL}}
\def\PSL{\operatorname{PSL}}
\def\li{\operatorname{li}}
\def\sym{\operatorname{sym}}

\def\Mob{M{\"o}bius }

\def\fF{\mathfrak{F}}
\def\M{\mathsf {M}}
\def\T{\mathsf {T}}

\def\e{{\mathbf{\,e}}}
\def\ep{{\mathbf{\,e}}_p}
\def\eq{{\mathbf{\,e}}_q}

\def\\{\cr}
\def\({\left(}
\def\){\right)}
\def\fl#1{\left\lfloor#1\right\rfloor}
\def\rf#1{\left\lceil#1\right\rceil}

\def\Tr{{\mathrm{Tr}}}
\def\Nm{{\mathrm{Nm}}}
\def\Im{{\mathrm{Im}}}

\def \oF {\overline \F}

\newcommand{\pfrac}[2]{{\left(\frac{#1}{#2}\right)}}

\def \Prob{{\mathrm {}}}
\def\e{\mathbf{e}}
\def\ep{{\mathbf{\,e}}_p}
\def\epp{{\mathbf{\,e}}_{p^2}}
\def\em{{\mathbf{\,e}}_m}

\def\Res{\mathrm{Res}}
\def\Orb{\mathrm{Orb}}

\def\vec#1{\mathbf{#1}}
\def \va{\vec{a}}
\def \vb{\vec{b}}
\def \vc{\vec{c}}
\def \vh{\vec{h}}
\def \vm{\vec{m}}
\def \vn{\vec{n}}
\def \vs{\vec{s}}
\def \vu{\vec{u}}
\def \vv{\vec{v}}
\def \vx{\vec{x}}
\def \vy{\vec{y}}
\def \vz{\vec{z}}

\def\flp#1{{\left\langle#1\right\rangle}_p}
\def\T {\mathsf {T}}

\def\mand{\qquad\mbox{and}\qquad}

 \date{\today}



\title[Mean value theorems for rational exponential sums]
{Mean value theorems for rational exponential sums}


\author[D. Koh]{Doowon Koh}

\address{Department of Mathematics,
Chungbuk National University,
Cheon\-gju, Chungbuk 28644, Korea}
\email{koh131@chungbuk.ac.kr}

\author[I. E. Shparlinski] {Igor E. Shparlinski}
\address{School of Mathematics and Statistics, University of New South Wales, Sydney NSW 2052, Australia}
\email{igor.shparlinski@unsw.edu.au}

\begin{abstract} We obtain finite field analogues of a series of recent results on various 
 mean value  theorems for Weyl sums. 
Instead of the Vinogradov Mean Value Theorem, our results rest on the classical argument 
of Mordell, combined with several other ideas.
 \end{abstract}

\keywords{Rational exponential sums, mean values, restricted mean values}
\subjclass[2020]{11L07,  11T23}

\maketitle

\tableofcontents 

\section{Introduction}\label{sec:intro}

\subsection{Set-up} 
For a prime $p$, as usual we use $\F_p$ to denote the finite 
field of $p$ elements. 
Given a $d$-dimensional  vector $\va =(a_1, \ldots, a_d) \in \F_p^d$ 
and an integer $N \le p$ we consider the exponential sums 
\[
S(\va; N) = \sum_{n =1}^N \ep\(a_1n+ \ldots + a_dn^d\), 
\]
where $\ep(z) = \exp(2\pi i z/p)$ and we always assume that the elements 
of $\F_p$ are represented by the set $\{0, \ldots, p-1\}$. 

For complete sums, that is, for $N = p$, the celebrated  result  of Weil~\cite{Weil} implies  that  for any non-zero 
vector  $\va \in \F_p^d\setminus \{\mathbf{0} \}$,  over  a finite field $\F_p=\Z_p$ of $p$ elements, where $p$ is prime,   we have an essentially optimal estimate 
\begin{equation}
\label{eq:Weil}
\left| S(\va; p) \right| \le (d - 1) p^{1/2},
\end{equation}
see, for example,~\cite[Chapter~11]{IwKow} or~\cite[Chapter~6]{Li}. 

Using the  standard completing technique (see~\cite[Section~12.2]{IwKow}), 
one can extend~\eqref{eq:Weil} to incomplete sums with $N < p$ at the cost of just a logarithmic 
loss and show that 
\begin{equation}
\label{eq:Weil-Incompl}
\left|S(\va; N) \right| \ll p^{1/2} \log p,
\end{equation}
where as usual, the notations  $A \ll B$ and $A = O(B)$, are equivalent
to $|A|  \le c B$ for some constant $c>0$,
which throughout this work  may depend only  on the degree $d \ge 2$.

Note that while the bound~\eqref{eq:Weil-Incompl} is satisfactory for large $N$,  it loses its strength when 
$N$ decreases and it becomes trivial when $N \le p^{1/2}$. Unfortunately, for smaller values of $N$, the
only general alternative to~\eqref{eq:Weil-Incompl} is a general bound on Weyl sums, 
which is a direct implication of the current form of the  {\it Vinogradov mean value theorem} ~\cite{BDG, Wool-Nest}
 and is  given in~\cite[Theorem~5]{Bourg1}. We note however that in a special case of monomial sums 
 there is a stronger result of Kerr~\cite{Kerr}. 
 
 On the other hand,  the classical argument of Mordell~\cite{Mor} can easily be generalised to give the 
 following rather precise formula  
 \begin{equation}
\label{eq:Mord}
M_{d,2s}(N) = s! N^s + O\(N^{s-1}\),  \qquad s =1, \ldots, d, 
\end{equation}
for the first $d$ even moments of the above sums, where
 \[
 M_{d,\rho} (N) = \frac{1}{p^d} \sum_{\va \in \F_p^d} |S(\va; N) |^\rho . 
\]

Thus motivated by the discrepancy between rather limited pointwise estimates on the sums $S(\va; N)$
and very precise value on their moments, one can consider scenarios which ``interpolate'' between these 
two extreme cases and consider bounding $S(\va; N)$ with a limited amount of averaging over $\va$.

For example, given a non-empty set $\fU \subseteq \F_p^d$
of cardinality $U = \# \fU$
we consider the {\it restricted\/} average 
\begin{equation}
\label{eq:MU}
 M_{d,k} (\fU; N) = \frac{1}{U} \sum_{\va \in\fU} |S(\va; N) |^k .
\end{equation}
Interesting examples of such sets $\fU$ include cubes 
\begin{equation}
\label{eq:Cube}
\fB = [A_1+1, A_1+H]\times \ldots \times [A_d +1, A_d + H]
\end{equation}
with integers $A_1, \ldots, A_d$ and $H \le p$ (or more general 
rectangular boxes) and moment curves 
\begin{equation}
\label{eq:MomCurve}
\fM = \{(t, \ldots, t^d):~ t \in \F_p\}.
\end{equation}

 We note that  $M_{d,k} (\fU; N)$   is a discrete analogue of similarly defined 
quantities  in  the continuous setting of Weyl sums, 
see~\cite{Bourg2, BrChSh, ChSh-QJM, ChSh-AMH,  DGW, DeLa, GuMa, MaOh, OhYe} 
and  references therein.  
Several more results  on the distribution of small and large values of Weyl sums can be found here~\cite{BaChSh2,CKMS, ChSh-JMAA}.

 \subsection{Main results} 
In order to study   $M_{d,k} (\fU; N)$   we  need to obtain new estimates on 
 $M_{d,k} (N)$,  which go beyond~\eqref{eq:Mord} and its natural implication (following from the H{\"o}lder
 inequality), asserting that 
  \begin{equation}
\label{eq:Mord s d}
M_{d,2s}(N) \ll N^s + N^{2s - d} , 
\end{equation}
which holds for any real $s \ge 1$. We emphasise that  $s$ in~\eqref{eq:Mord s d}  does not have to be 
integer. 

We now obtain another bound, which is better than~\eqref{eq:Mord s d}  for large values of $s$.
We prove it under the condition 
\begin{equation}
\label{eq:Nk~p}
sN^k < p \le sN^{k+1}
\end{equation}
for some integer $k \ge 1$ 
(note that we discard the case $N \ge p/s$ as  for such $N$
 the pointwise  bound~\eqref{eq:Weil-Incompl} is essentially optimal).

\begin{thm}  \label{thm:MVT-Mod_p large s} Let $k$ be a fixed integer with $d\ge k \ge 1$. 
For any  $N$ satisfying~\eqref{eq:Nk~p} and  for any positive integer $s \ge d(d+1)/2$
 we have
\[
 M_{d,2s}(N) \le N^{2s - k(k+1)/2+ o(1)} p^{-d+k}. 
 \]
\end{thm}

 Note that the saving of Theorem~\ref{thm:MVT-Mod_p large s}  improves that of~\eqref{eq:Mord s d} for 
 all $N \le p^{1-\kappa}$ with some fixed $\kappa > 0$. Indeed, we need to show  that   under this 
 condition we have 
 \[
N^{k(k+1)/2} p^{d-k} \ge N^{d+\eta}
 \]
 for some fixed $\eta$
 or 
\begin{equation}
\label{eq:Improve}
N^{d - k(k+1)/2 + \eta}  \le p^{d-k}.
\end{equation}
If $k =1$  then~\eqref{eq:Improve} is guaranteed if
$N \le p^{1-\kappa}$ for some $\kappa > 0$ (with a suitable $\eta > 0$ depending on $\kappa$).
Recalling~\eqref{eq:Nk~p}, we see that if  $k \ge 2$  then~\eqref{eq:Improve} is guaranteed if
$d/k-(k+1)/2   <  d-k $ or $k/2 < d$, which is much weaker than the condition $k \le d$ which we impose in 
Theorem~\ref{thm:MVT-Mod_p large s}. 

\begin{rem}\label{rem:large k$}
It is easy to see that from the proof of Theorem~\ref{thm:MVT-Mod_p large s}  that for $k \ge d$ we have 
$M_{d,2s}(N)  =  J_{d,s}(N)$ and hence we have an essentially optimal bound from Lemma~\ref{lem:MVT}. 
\end{rem}

Next we improve~\eqref{eq:Mord s d} for smaller values of $s$ 

\begin{thm}  \label{thm:MVT-Mod_p med s} Let $k$ be a fixed integer with $d-2\ge k \ge 1$. 
For any  $N$ satisfying~\eqref{eq:Nk~p} and  for any positive integer $s \le (k+1)(k+2)/2-1$
 we have
\[
 M_{d,2s}(N) \le  N^{d+s+o(1)}p^{-1} + N^{s+ d(d+1)/2 -  k(k+1)/2- 1+o(1)}~p^{-d+k}. 
 \]
\end{thm}

\begin{rem} 
Our proof of Theorem~\ref{thm:MVT-Mod_p med s} relies on a result of Wooley~\cite[Theorem~1.2]{Wool-Inhom}.
Using~\cite[Theorems~1.1 and~1.3]{Wool-Inhom} in our argument one can get other versions of Theorem~\ref{thm:MVT-Mod_p med s}. 
\end{rem}

We now establish the following bound $M_{d,k} (\fB; N)$, where $\fB$ is a cube as in~\eqref{eq:Cube}.

\begin{thm}  \label{thm:small_box}
Let   $H, N\le p$ be positive integers. Then, for any cube $\fB$  of the form~\eqref{eq:Cube},  and for any positive integer $s < d,$
 we have
\begin{align*}
 M_{d,2s} (\fB; N) & \ll  \min\left\{   pH^{-1} N^{s},   p^{1/2} H^{-1/2}  \(N^{s} + N^{2s-d/2}\)\right\} \\
 & \qquad \quad + \min\left\{(p/H)^d N^{ s-1/2} ,    (p/H)^{d/2}   \(N^s + N^{2s-d/2}\)\right\}.
\end{align*}
 \end{thm}

\begin{rem}\label{rem: Gen Sums}
We actually establish  Theorem~\ref{thm:small_box} for moment of more general 
weighted rational sums, see the definition~\eqref{eq:Weight Mon}.
\end{rem}

We note that Theorem~\ref{thm:small_box}  is always better that the trivial bound 
\[
 M_{d,2s} (\fB; N)  \le (p/H)^d  M_{d,2s} (N) \le  (p/H)^d  N^s.
 \]
 We now compare Theorem~\ref{thm:small_box}   with another trivial bound 
\begin{equation}
\label{eq:Triv N2s}
 M_{d,2s} (\fB; N) \le  N^{2s}.
\end{equation}

We observe that for $s\ge d/2$ the bound of Theorem~\ref{thm:small_box}
simplifies as    
\begin{align*}
 M_{d,2s} (\fB; N) & \ll  \min\left\{   pH^{-1} N^{s},   p^{1/2} H^{-1/2} N^{2s-d/2} \right\} \\
 & \qquad \quad + \min\left\{(p/H)^d N^{ s-1/2} ,    (p/H)^{d/2} N^{2s-d/2} \right\}. 
\end{align*} 
Clearly the bound improves the trivial bound~\eqref{eq:Triv N2s} provided that
\[
pH^{-1} \le  \max\left\{N^{ s}, N^{d}\right\}  =N^d
\] \
and  
\[(p/H)^d \le \max\left\{N^{ s+1/2}, N^{d}  \right\}  =N^d, 
\]
which, using that $d/2 \le s \le d-1$,  can be combined as 
\[
HN \ge p.
\]

Similarly, for $s < d/2$,  we see that Theorem~\ref{thm:small_box} yields 
the bound
\begin{align*}
 M_{d,2s} (\fB; N) & \ll  \min\left\{   pH^{-1} N^{s},   p^{1/2} H^{-1/2}  N^{s} \right\} \\
 & \qquad \quad + \min\left\{(p/H)^d N^{ s-1/2} ,    (p/H)^{d/2}  N^{s} \right\}\\
 & = p^{1/2} H^{-1/2}  N^{s}+ \min\left\{(p/H)^d N^{ s-1/2} ,    (p/H)^{d/2}  N^{s} \right\}, 
\end{align*}   
which, as it  is easy to check,  remains nontrivial under the same condition
\[
HN^{2s/d} \ge p.
\]

\section{Bounds on the number of solutions to some systems of  equations and congruences}\label{sec:syst cong eq}

\subsection{Preliminaries bounds}
For $\vh = (h_1, \ldots, h_d) \in \Z^d$, we denote $J_{d,s}(N; \vh)$  as the number of solutions to the system 
of  equations 
\begin{align*}
    n_{1}+\ldots+n_{s}&= n_{s+1}+\ldots+n_{2s}  + h_1, \\
    n^{2}_{1}+\ldots+n^{2}_{s}&= n^{2}_{s+1}+\ldots+n^{2}_{2s}  +h_2, \\
    &\vdots \\
    n^{d}_{1}+\ldots+n^{d}_{s}&= n^{d}_{s+1}+\ldots+n^{d}_{2s} + h_d, \end{align*}
 in  integer variables $1\le n_i  \le N$,  $i = 1,\ldots, 2s$. 

 For the zero vector $\vh =  \vec{0}$ we denote $J_{d,s}(N) = J_{d,s}(N;\vec{0})$. 
 
 The following bound is known as the optimal form of the Vinogradov Mean Value 
 theorem and is due to  Bourgain, Demeter and Guth~\cite{BDG} and via an alternative 
 approach to Wooley~\cite{Wool-Nest}. 
 
 \begin{lem} \label{lem:MVT}
For any integer $s$ we have 
\[
J_{d,s}(N)  \le \(N^s + N^{2s - d(d+1)/2}\) N^{o(1)}.
\]
 \end{lem}
 
 We note that the bound of Lemma~\ref{lem:MVT} also applies to  $J_{d, s}(N; \vh)$
 with any  $\vh   \in \Z^d$,  that is 
\begin{equation}
\label{eq:Triv Jh}
 J_{d,s}(N; \vh) \le  J_{d,s}(N) 
\end{equation}
 for any $\vh   \in \Z^d$,  see, for example,~\cite[Equation~(13)]{CGOS} (the proof 
 is fully analogous to the derivation of~\eqref{eq:Triv Rh} below). 
 However, as it is shown first by Brandes and K. Hughes~\cite{BrHu} 
 and then  in a stronger form by Wooley~\cite{Wool-Inhom} for $s < d(d+1)$, one can 
 take advantage of the property $\vh \ne  \vec{0}$.
 
 We do not present all results of this form but rather recall one of them given 
 by~\cite[Theorem~1.2]{Wool-Inhom}.

 \begin{lem} \label{lem:MVT-Inhom-Z}
 Let $d \ge 3$ and let  $\ell<d$ be  the smallest index having
the property that $h_\ell \ne 0$, where $\vh =(h_1, \ldots, h_d) \in \Z^d$. 
Then for any positive integer  $s \le \ell(\ell + 1)/2$, one has
\[
 J_{d,s}(N; \vh) \le  N^{s-1+o(1)}. 
 \]
 \end{lem}

Similarly to the above, for $\vh = (h_1, \ldots, h_d) \in \F_p^d$, we define $R_{d,s}(N; \vh)$ as the number of solutions to the system of congruences.   
\[
    n^{\nu}_{1}+\ldots+n^{\nu}_{s}\equiv n^{\nu}_{s+1}+\ldots+n^{\nu}_{2s} +h_{\nu} \pmod p,
    \qquad \   \nu =1, \ldots,  d,
\]
 in  integer variables $1\le n_i  \le N$,  $i = 1,\ldots, 2s$.

It is easy to see that  the orthogonality of exponential functions allows us to express  $R_{d,s}(N; \vh)$
via exponential sums as 
\[
R_{d,s}(N; \vh) =  \frac{1}{p^d} \sum_{\va \in \F_p^d} |S(\va; N) |^{2s} \ep\(-\langle \va, \vh\rangle\),
\]
where $\langle \va, \vh\rangle$ denotes the inner product. 

In particular, this immediately implies that 
\begin{equation}
\label{eq:Triv Rh}
R_{d,s}(N; \vh)  \le  M_{d,2s}(N)  =  R_{d,s}(N) ,
\end{equation}
where $R_{d,s}(N) = R_{d,s}(N, \vec{0})$ and, as before, $\vec{0}$ denotes the zero vector. 
We note that using a similar integral representation for $J_{d,s}(N; \vh)$, we obtain a proof
of~\eqref{eq:Triv Jh}.

\subsection{New subconvexity bounds}
We now use some ideas of Brandes and Hughes~\cite{BrHu}, in a more refined form given by Wooley~\cite{Wool-Inhom}  
to obtain a more precise bound than~\eqref{eq:Triv Rh}.

First we formulate a  discrete analogue of~\cite[Lemma~2.1]{BrHu}  
combined with a discrete analogue of~\cite[Equation~(2.5)]{BrHu}, see 
also~\cite[Lemma~2.1]{Wool-Inhom}, for a more general statement, which we do 
not need. We omit the proof as it follows exactly the 
same steps and calculations as in~\cite{BrHu}, after one replaces integrals over the unit cube by sums 
of $\F_p^d$. We need to introduce some notation first. For $\vh = (h_1, \ldots, h_d) \in \F_p^d$,
we consider polynomials 
$$
\psi_\nu(\vh; X) = \sum_{i=0}^{\nu-1} \binom{\nu}{i} h_{\nu-i} X^i, \qquad \nu= 1, \ldots, d. 
$$

Then, given a $d$-dimensional  vector $\va =(a_1, \ldots, a_d) \in \F_p^d$
and an integer $N \le p$ we consider the exponential sums 
\[
U(\va, \vh;N) = \sum_{u =1}^N \ep\( a_1 \psi_1(\vh; u) + \ldots + a_d \psi_d(\vh; u)\). 
\]

 \begin{lem} \label{lem: R SU}
For any integer $s\ge 1$ we have 
\[
R_{d,s}(N; \vh) \ll \frac{1}{N p^{d}}   \sum_{\va \in \F_p^d} |S(\va; 2N) |^{2s} U(-\va; \vh; N) . 
\]
 \end{lem}
 
 The following bound is an analogue of~\cite[Lemma~3.1]{Wool-Inhom}, with a fully analogous proof 
 where we use~\eqref{eq:Mord s d} (similarly to how Lemma~\ref{lem:MVT} has been used in~\cite{Wool-Inhom}). 
 Hence we just sketch the argument.

  \begin{lem} \label{lem:MVT_SU}  Let $d \ge 3$ and let  $\ell<d$ be  the smallest index having
the property that $h_\ell \ne 0$, where $\vh =(h_1, \ldots, h_d) \in \Z^d$. 
Then for any positive integers $s \le \ell$ and $t$ we have 
\[
  \sum_{\va \in \F_p^d} |S(\va; 2N) |^{2s} |U(\va; \vh; N)|^{2t} \ll  p^dN^s\(N^{t}  + N^{2t - d+ \ell}\). 
\]
 \end{lem}
 
 \begin{proof} By the orthogonality of exponential functions we have
\[
  \sum_{\va \in \F_p^d} |S(\va; 2N) |^{2s} |U(\va; \vh; N)|^{2t} = p^d \fI,
 \]
 where $\fI$ is  the number of solutions to the system of congruences:
 \begin{equation}
\label{eq:System-psi}
\begin{split}
 \sum_{i=1}^s  \(n_{i}^\nu -  n_{s+i}^\nu\)  \equiv    \sum_{j=1}^t  & \left(\psi_\nu(\vh; u_j) -  \psi_\nu(\vh; u_{t+j} ) \right)\pmod p, \\
 \nu &=1, \ldots, d, 
 \end{split}
 \end{equation}
 in  integer variables $1\le n_i  \le 2N$,  $i = 1,\ldots, 2s$, and
  $1\le u_j \le N$,  $j= 1,\ldots, 2t$. 
  
  Since $h_1 = \ldots = h_{\ell-1} = 0$ we see that the polynomials 
  $\psi_j(\vh; X)$ are identical to zero for $j = 1, \ldots, \ell-1$ while 
 $ \psi_\ell(\vh; X) = h_\ell$ is a constant polynomial.  Hence we derive 
 from~\eqref{eq:System-psi} that 
\[
 \sum_{i=1}^s  \(n_{i}^\nu -  n_{s+i}^\nu\) \equiv  0  \pmod p, \qquad  \nu =1, \ldots, \ell,
\]
 and thus by our assumption $s \le \ell$ we see from~\eqref{eq:Mord s d}
 that there are $O(N^s)$ choices for $\(n_1, \ldots, n_{2s}\)$. 

Once the vector $\vn = \(n_1, \ldots, n_{2s}\)$  is fixed, the same algebraic manipulations as
in the proof of~\cite[Lemma~3.1]{Wool-Inhom} leads to the system of congruences 
\[
 \sum_{i=1}^t  \(u_{i}^\nu -  u_{t+i}^\nu\) \equiv  z_\nu  \pmod p, \qquad  \nu =1, \ldots,  d-\ell, 
\]
with some integers $z_1, \ldots, z_{d-\ell}$ depending only on $\vh$ and $\vn$. 
Therefore, by~\eqref{eq:Mord s d}  and~\eqref{eq:Triv Rh} we have 
$O\(N^{t}  + N^{2t - d+\ell}\)$ choices for $(u_1, \ldots, u_t)$. Hence 
$\fI \ll N^s\(N^{t}  + N^{2t - d+\ell}\)$, which concludes the proof.  \end{proof} 

Next we establish an analogue of~\cite[Theorems~1.1 and~1.3]{Wool-Inhom}.

We observe that in the case of rational sums we have the same saving $1/2$ against 
the trivial bound
\begin{equation}\label{eq:Rs0}
R_{d, s} (N;  \vh) \ll N^s,
\end{equation}
 implied for $s< d$ by~\eqref{eq:Mord s d} and~\eqref{eq:Triv Rh}. We note that this saving applies
 in the full range $s<d$, which shows its distinctions from   the case for Weyl sums, 
see~\cite[Theorems~1.1 and~1.3]{Wool-Inhom}.

  \begin{lem} \label{lem:MVT-Inhom-Mod_p}  Let $d \ge 3$ and let   the smallest index  $\ell$ 
with   $h_\ell \ne 0$, where $\vh =(h_1, \ldots, h_d) \in \Z^d$, satisfies  $\ell<d$. 
Then for any positive integer  $s < d$   we have 
\[
R_{d,s}(N; \vh) \ll  N^{s-1/2}. 
\] 
\end{lem} 
 
 \begin{proof} We  set 
 \[
r = 2(d- \ell), \qquad u = \min\{rs, \ell\},  \qquad  v = \frac{rs - u}{r-1}.
\]
Note that $r\ge 2$ is an even integer and  $u, v\ge 0$ with
\[
\frac{2u}{r} + 2v\frac{r-1}{r} = 2s.
\]
 
Hence, using Lemma~\ref{lem: R SU},   we have 
 \begin{equation}
\label{eq:RW1W2}
 R_{d,s}(N; \vh) \ll \frac{1}{N p^{d}}   W_1^{1/r}W_2^{1-1/r}, 
 \end{equation}
 where 
 \begin{align*}
&W_1 =  \sum_{\va \in \F_p^d} |S(\va; 2N) |^{2u} |U(\va; \vh; N)|^r, \\
& W_2 =  \sum_{\va \in \F_p^d} |S(\va; 2N) |^{2v}  . 
\end{align*}
Using Lemma~\ref{lem:MVT_SU} (which applies since $u \le \ell$ and $r$ is an even integer)  
and the bound~\eqref{eq:Mord s d}, we derive 
\[
W_1 \ll p^dN^{u + r/2} \mand W_2 \ll p^d\(N^{v} + N^{2v-d}\), 
\]
respectively. Substituting this bound in~\eqref{eq:RW1W2} yields
 \begin{equation}
 \begin{split}
\label{eq:R-bound}
 R_{d,s}(N; \vh) & \ll \frac{1}{N}  N^{u/r+ 1/2 + v(1-1/r)}\(1 +   N^{v-d}\)^{{1-1/r}}\\
& =  N^{s-1/2} \(1 +   N^{v-d}\)^{1-1/r}. 
\end{split}
 \end{equation}  

We now observe that $v \le d$, which is  equivalent to 
 \begin{equation}
\label{eq:u large}
u\ge rs - d (r-1).
 \end{equation}
 
The inequality~\eqref{eq:u large} is  obvious if $u= rs$, so we can only consider $u = \ell$, in 
which case~\eqref{eq:u large}  is equivalent to  
\[
\ell \ge 2(d-\ell) s - d\(2(d-\ell)-1\) = -2(d-\ell)(d-s) + d.  
\]
In turn, this  is equivalent to 
\[
\(d-\ell\)\(2(d-s) - 1\)\ge 0
\]
which always holds for $\ell \ge d$ and $s < d$. Thus the bound~\eqref{eq:R-bound} yields the desired result. 
 \end{proof} 
 
 \subsection{Restricted mean value theorems and congruences}
 
 We need the following discrete analogue of~\cite[Lemma~3.8]{CKMS}. 

\begin{lemma}
\label{lem:MVT2Ineq}
  Let $I$ be an interval and $\varphi_1,\varphi_2,\ldots \varphi_{d}$ real valued functions on $I$. For any   integers $s\ge 1$ and positive integers  $H_1,\ldots,H_d < p$ 
and sequence of complex numbers $\gamma_n$ satisfying $|\gamma_n|\le 1$ we have 
\[
\frac{1}{H_1\ldots H_d}\sum_{a_1=1}^{H_1}\ldots \sum_{a_d=1}^{H_d}\left|\sum_{n\in I} 
\gamma_n \ep\left(\sum_{i=1}^{d}a_i\varphi_i(n) \right)\right|^{2s} \ll \sfT, 
\]
where $\sfT$ is the number of solutions to the system of $d$ congruences 
\[
\sum_{j=1}^{2s} (-1)^j \varphi_i\(n_j\) \equiv u_i  \pmod p, 
\qquad i =1, \ldots, d,
\]
in integer variables $ n_1,\ldots,n_{2s}\in I$, $ u_i \in [-p/H_i, p/H_i]$,  $i =1, \ldots, d$. 
\end{lemma}

\begin{proof}
Let $F$ be a positive smooth function with sufficient decay satisfying  
$$
F(x)\gg 1 \ \ \text{if} \ \  |x|\le 1 \mand   \supp \widehat F\subseteq [-1,1], 
$$
where $  \supp \widehat F = \{x\in \R:~ \widehat F(x) \ne 0\}$ and  
\[
 \widehat F(x) = \int_{-\infty}^\infty F(y) \e(-xy) dy
\]
is the Fourier transform of $F$ and $\e(z) = \exp(2 \pi i z)$. 
A concrete examples is given by 
\[
F(x) = \begin{cases}\displaystyle{ \(\frac{ \sin \pi x}{\pi x}\)^2}, & \text{if}\ x \ne 0,\\
1, &  \text{if}\ x = 0,
\end{cases}
\]
see the proof of~\cite[Lemma~2.1]{Watt}. 

By Poisson summations, for any $\gamma \in \R$, 
\begin{align*}
\sum_{a \in \Z} F(a/H) \e(\gamma a)& = H \sum_{a \in \Z} \widehat F\((a-\gamma) H\)\\
& \ll \# \{n \in \Z: ~|n-\gamma| \le H^{-1}\} 
=  \begin{cases}  1, & \text{if} \ \|\gamma\| \le H^{-1},\\
0, & \text{otherwise}. 
\end{cases}
\end{align*}
where 
\[
 \|\gamma\| = \min_{k \in \Z} |\gamma - k|
 \]
is the distance between $\gamma$ and the closest integer. 

Expanding the $2s$-th power, interchanging summation, recalling the assumption $|\alpha_n|\le1$, 
this immediately implies the desired result. 
\end{proof}


\section{Proofs of main results}

\subsection{Proof of Theorem~\ref{thm:MVT-Mod_p large s}}

We recall that $M_{d,2s}(N)  = R_{d,s}(N)$, see~\eqref{eq:Triv Rh}. 
In fact from the condition~\eqref{eq:Nk~p} we see that the congruences corresponding 
to the powers $j =1, \ldots, k$ are in fact equations. Furthermore,  for $j = k+1, \ldots, d$
the corresponding congruence can be written as 
\[
 n^{j}_{1}+\ldots+n^{j}_{s} - n^{j}_{s+1}+\ldots+n^{j}_{2s}  = m_j p
 \]
 with some integers $m_j$ satisfying 
\begin{equation}
\label{eq:mj}
|m_j| \le sN^j/p, \qquad j = k+1, \ldots, d.
\end{equation}

Therefore we now see that 
\begin{equation}
\label{eq:R vs Jmj}
R_{d,s}(N)  \le \sum_{\substack{\vm = (0, \ldots, 0, pm_{k+1}, \ldots, pm_d) \in \Z^{d}\\|m_j| \le sN^j/p, \  j = k+1, \ldots, d}}
 J_{d,s}(N; \vm) .
\end{equation}

We see from~\eqref{eq:mj} and the condition~\eqref{eq:Nk~p}  that the number 
for all possible choices of $\vm$ is dominated by
\[
\prod_{j=k+1}^d \(2sN^j/p +1\) \ll p^{-d+k}  N^{d(d+1)/2   -  k(k+1)/2}.
\]  

Combining~\eqref{eq:R vs Jmj} with~\eqref{eq:Triv Jh} and using Lemma~\ref{lem:MVT}, 
after simple calculations we obtain the result.

\subsection{Proof of Theorem~\ref{thm:MVT-Mod_p med s}}

We appeal to the bound~\eqref{eq:R vs Jmj} again. 

Let  $U_\ell$, $\ell =k+1, \ldots, d$,  be the contribution to the sum in~\eqref{eq:R vs Jmj}  from the non-zero 
vectors  $\vm = (0, \ldots, 0, pm_{k+1}, \ldots, pm_d ) \in \Z^d$ for which $\ell$ is the smallest index having
the property that $m_\ell \ne 0$, and such $\ell$ exists. Clearly the contribution from the zero vector $\vm = \mathbf 0$ 
is  $J_{d,s}(N)$. Furthermore, we use~\eqref{eq:Triv Jh} and Lemma~\ref{lem:MVT} to estimate 
$U_d$, and use  Lemma~\ref{lem:MVT-Inhom-Z} to estimate  $U_\ell$ for $k+1 \le \ell \le d-1$. This yields the bound 
\begin{align*}
R_{d,s}(N)  & \le N^{d+s+o(1)}p^{-1} + \sum_{\ell = k+1}^{d-1} N^{d(d+1)/2 - \ell(\ell-1)/2+s-1+o(1)}p^{-d+\ell-1}\\
 & = N^{d+s+o(1)}p^{-1} + N^{d(d+1)/2 +s-1+o(1)}p^{-d-1}
  \sum_{\ell = k+1}^{d-1} N^{ - \ell(\ell-1)/2}p^{\ell}. 
\end{align*}

To analyse the last sum we define the function we notice that its terms are of monotonically decreasing
order.
Indeed, we note that 
\[
 N^{ - \ell(\ell-1)/2}p^{\ell} \gg  N^{ -  \ell (\ell+1)/2}p^{\ell+1}
 \]
 is equivalent to $N^{\ell}  \gg  p$,  which holds for $\ell \ge k+1$ by the choice of $k$. 
 Hence we now conclude that 
 \[
 R_{d,s}(N)   \le   N^{d+s+o(1)}p^{-1} + N^{d(d+1)/2 +s - k(k+1)/2- 1+o(1)}p^{-d+k},
\]
which by~\eqref{eq:Triv Rh} is equivalent to the desired statement. 

\subsection{Proof of Theorem~\ref{thm:small_box}}

Given a  sequence of complex numbers $\bgamma = \(\gamma_n\)_{n=1}^N$ satisfying $|\gamma_n|\le 1$,  we 
define more general sums 
\[
S(\va, \bgamma; N) = \sum_{n =1}^N \gamma_n \ep\(a_1n+ \ldots + a_dn^d\), 
\]
and study their moments 
\begin{equation}
\label{eq:Weight Mon}
 M_{d,2s} (\fB, \bgamma; N) = \frac{1}{H^d} \sum_{\va \in\fB} |S(\va, \bgamma; N) |^{2s} ,
\end{equation}
where now without loss of generality we can assume that 
$\fB = [1, H]^d$.

We follow the ideas of~\cite[Section~5]{BrChSh} and define 
\[
\cU = [-p/H, p/H]^d \cap \Z^d  
\]

 We have the obvious partition 
\[
	\cU  = \{\bm 0\} \cup \cG \cup \cH,
\]
where  
\[
	\cG = \{\vh \in \cU:~\vh = (0,\ldots,0,   h_d), \; h_d \neq 0\}.
\]
and $\cH$ contains all remaining vectors $\vh \in \cU $, that is, 
\[
\cH = \cU \setminus \(\{\bm 0\} \cup \cG \)
\]

So,  by Lemma~\ref{lem:MVT2Ineq} 
\begin{equation}\label{eq:mv-partition}
\begin{split} 
 M_{d,2s} (\fB, \bgamma; N) & \ll  \sum_{\vh \in \cU} R_{d,s}(N; \vh)\\
& = R_{d,s} (N) + \sum_{\vh \in \cG} R_{d,s} (N; \vh) + \sum_{\vh \in \cH} R_{d,s} (N; \vh).
\end{split} 
\end{equation}

For the first term  $R_{d, s} (N)$ we use the bound~\eqref{eq:Rs0}.
Namely, we obtain that  for $s<d$, 
\begin{equation} \label{E1K} R_{d, s}(N) \ll N^s. \end{equation}
To estimate the contributions from the other terms, 
 we  record the following analogue of~\cite[Lemma~5.3]{BrChSh}, 
which follows from the Cauchy inequality. Namely, for $d \ge 2$
and  any finite set $\cS \subseteq \F_p^d$ we have
\begin{equation}\label{eq: Inhom VMVT-Cauchy}
		\sum_{\vh \in \cS} R_{d, s} (N; \vh) \le \(\# \cS R_{d, 2s} (N)\)^{1/2}. 
\end{equation}

 Next, we observe that 
\[
	\# \cG \ll  p/H \mand \# \cH \ll   (p/H)^d. 
\]

Hence we see from~\eqref{eq:Rs0} 
(which takes the form $M_{d,2s}(N) \ll N^s$ for $s\le d$) and~\eqref{eq: Inhom VMVT-Cauchy} together with~\eqref{eq:Mord s d} 
and~\eqref{eq:Triv Rh} that 
\begin{equation}\label{eq: Contr G}
\begin{split}
		\sum_{\vh \in \cG} R_{d,s} (N; \vh) 
		& \ll  \min\left\{\# \cG N^s,   \(\#\cG\(N^{2s} + N^{4s-d}\)\)^{1/2} \right\}\\
		& \ll  \min\left\{   pH^{-1} N^{s},   p^{1/2} H^{-1/2}  \(N^{s} + N^{2s-d/2}\)\right\}.
		 \end{split}
\end{equation}

Similarly, using Lemma~\ref{lem:MVT-Inhom-Mod_p} and  the bound~\eqref{eq:Mord s d}, we see that that 
\begin{equation}\label{eq: Contr H}
\begin{split}
\sum_{\vh \in \cH}  R_{d,s} & (N; \vh) \\
& \ll 
\min\left\{(p/H)^d N^{ s-1/2} ,    (p/H)^{d/2}   \(N^s + N^{2s-d/2}\)\right\}.
\end{split}
\end{equation}

Thus, substituting~\eqref{E1K}, \eqref{eq: Contr G} and~\eqref{eq: Contr H}  in~\eqref{eq:mv-partition} 
and observing that the bound~\eqref{eq: Contr G}  
dominates~\eqref{E1K}, we derive 
\begin{align*}
 M_{d,2s} (\fB, \bgamma; N) & \ll  \min\left\{   pH^{-1} N^{s},   p^{1/2} H^{-1/2}  \(N^{s} + N^{2s-d/2}\)\right\} \\
 & \quad \quad + \min\left\{(p/H)^d N^{ s-1/2} ,    (p/H)^{d/2}   \(N^s + N^{2s-d/2}\)\right\}, 
\end{align*}
which concludes the proof.

 \section{Comments}
 
Some of the tools and ideas of this paper can also be used to study a  variety of other 
problems on rational exponential sums.  Namely, 
given  a non-negative integer $k \le d$ and a permutation $\pi \in \sfS_d$ of the symmetric group of permutations on 
$d$ symbols and vectors $\vb = \(b_1, \ldots, b_k\) \in \F_p^k$ and 
 $\vc = \(c_1, \ldots, c_{d-k}\) \in \F_p^{d-k}$ we define the {\it maximal\/} operator
 \[
 T_{d,\pi, k}(\vb; N) = \max_{\vc \in  \F_p^{d-k}}\left|  \sum_{n =1}^N 
 \ep\(\sum_{i=1}^k  b_in^{\pi(i)} + \sum_{j=1}^{d-k} c_j n^{\pi(k+j)} \)
 \right|. 
 \]
 A similar quantity has been studied quite extensively in the context of the classical Weyl sums, 
 see~\cite{Bak,Barr,BaChSh1, BPPSV, BrSh, ChSh-IMRN, ChSh-MMJ, ErdSha} for bound on maximal operators, 
as well as the references therein. Investigating the above maximal operators $ T_{d,\pi, k}(\vb; N)$ 
for 
rational exponential sums is a very interesting and natural question.  Even the case of the identical permutation $\pi(i) = i$, $i=1, \ldots, d$, is already interesting. 

One can also study averaging over parametric curves. For example, given $d$ polynomials $f_1, \ldots, f_d \in \F_p[X]$ one can ask about the average values $M_{d,k} (\fU; N)$ given by~\eqref{eq:MU} for parametric curves
\[
\fU =  \{(f_1(t), \ldots, f_d(t)):~ t \in \F_p\},
\]
in particular for the moment curve~\eqref{eq:MomCurve}.

 \section*{Acknowledgement}

The authors would like to thank Bryce Kerr for a suggestion to use Poisson summation in the 
proof of Lemma~\ref{lem:MVT2Ineq}. 

This work started during a very enjoyable visit of both authors to  Seoul National University whose 
 hospitality and support  are gratefully acknowledged. 
 During the preparation of this work, D.K. was supported by the National Research Foundation of Korea
(NRF) grant funded by the Korea government (MSIT) (NO. RS-2023-00249597)
and I.S. was supported, in part, by the Australian Research Council Grants DP230100530 and DP230100534.

\end{document}